\documentclass[final,3p,times]{elsarticle}
\usepackage{amsmath,amssymb,amsfonts,epsf,amsthm}
\usepackage{hyperref}
\hypersetup{colorlinks=true, urlcolor=blue, citecolor=cyan, pdfborder={0 0 0},}
\usepackage{soul}
\usepackage{natbib}
\usepackage{multirow}
\usepackage{amssymb}
\usepackage{amsmath,amssymb,amsfonts,epsf,amsthm}
\usepackage{hyperref}
\hypersetup{colorlinks=true, urlcolor=blue, citecolor=cyan, pdfborder={0 0 0},}
\usepackage{soul}
\usepackage{natbib}
\usepackage{multirow}
\usepackage{amsmath,amssymb,amscd,amsthm,verbatim,alltt,amsfonts,array}
\usepackage[english]{babel}
\usepackage{latexsym}
\usepackage{amssymb}
\usepackage{euscript}
\usepackage{graphicx}
\usepackage{pgf,tikz}
\usetikzlibrary{arrows}
\usepackage{mathrsfs}
\definecolor{zzttqq}{rgb}{0.6,0.2,0.}
\definecolor{qqqqff}{rgb}{0.,0.,1.}
\definecolor{xdxdff}{rgb}{0.49019607843137253,0.49019607843137253,1.}
\definecolor{qqffqq}{rgb}{0.,1.,0.}
\definecolor{ffffff}{rgb}{1.,1.,1.}
\definecolor{ffqqqq}{rgb}{1.,0.,0.}
\definecolor{ccffcc}{rgb}{0.8,1.,0.8}
\definecolor{ffffqq}{rgb}{1.,1.,0.}
\definecolor{qqffqq}{rgb}{0.,1.,0.}
\definecolor{uuuuuu}{rgb}{0.26666666666666666,0.26666666666666666,0.26666666666666666}
\definecolor{ccqqww}{rgb}{0.8,0.,0.4}
\definecolor{bfffqq}{rgb}{0.7490196078431373,1.,0.}
\definecolor{ffffff}{rgb}{1.,1.,1.}
\definecolor{ttqqqq}{rgb}{0.2,0.,0.}
\definecolor{qqffff}{rgb}{0.,1.,1.}
\definecolor{eqeqeq}{rgb}{0.8784313725490196,0.8784313725490196,0.8784313725490196}
\definecolor{uuuuuu}{rgb}{0.26666666666666666,0.26666666666666666,0.26666666666666666}
\definecolor{cqcqcq}{rgb}{0.7529411764705882,0.7529411764705882,0.7529411764705882}

\usepackage[ruled, vlined, Algorithm]{algorithm}
\usepackage{algorithmic}

\newtheorem{theorem}{Theorem}[section]

\newtheorem{proposition}{Proposition}[section]

\newtheorem{definition}{Definition}[section]

\numberwithin{equation}{section}

\usepackage{amssymb}

\journal{--}

\allowdisplaybreaks[1]

\begin{document}

\begin{frontmatter}
\title{  Computing Minimal Doubly Resolving Sets and the Strong Metric Dimension of the Layer Sun  Graph and the Line Graph of Layer Sun  Graph  }

\author[label1,label2]{Jia-Bao Liu}
\ead{liujiabaoad@163.com;liujiabao@ahjzu.edu.cn}
\author[label3]{Ali Zafari \corref{1}}
\ead{zafari.math.pu@gmail.com; zafari.math@pnu.ac.ir}
\address[label1]{ Department of Mathematics, Huainan Normal University, Huainan 232038,  P.R. China}
\address[label2]{School of Mathematics and Physics, Anhui Jianzhu University, Hefei 230601, P.R. China}
\address[label3]{Department of Mathematics, Faculty of Science, Payame Noor University, P.O. Box 19395-4697, Tehran, Iran}
\cortext[1]{Corresponding author}
\begin{abstract}
Let  $G$ be a finite, connected graph of order of at least 2, with vertex set $V(G)$  and edge set  $E (G)$.
A set $S$ of vertices of the graph $G$ is a doubly resolving set for $G$ if every two distinct vertices of $G$ are doubly resolved by some two vertices of $S$.  The minimal doubly resolving set of vertices of graph $G$ is a doubly resolving set with minimum cardinality and is denoted by  $\psi(G)$. In this paper, first, we construct a class of  graphs of order $2n+ \Sigma_{r=1}^{k-2}nm^{r}$, denoted by $LSG(n,m, k)$, and call these graphs as the layer Sun graphs  with parameters $n$, $m$ and $k$. Moreover, we  compute minimal doubly resolving sets and the strong metric dimension of layer Sun  graph $LSG(n,m, k)$ and the line graph of the layer Sun  graph $LSG(n,m, k)$.
\end{abstract}
\end{frontmatter}
\section{Introduction}
\label{sec:introduction}

In this paper, suppose $G$ is a finite, simple connected graph of order of, at least 2, with vertex set $V(G)$  and edge set  $E (G)$.
If $x$ and $y$ are vertices in the graph $G$, then the distance  $x$ from $y$ in $G$ is  denoted by $d_{G}(x, y)$ or simply  $d(x, y)$,
where $d(x, y)$ is the length of the shortest path from $x$ to $y$.
The line graph of a graph $G$ is denoted by $L(G)$, with vertex set $V(L(G))=E (G)$ and where two edges of $G$ are adjacent in $L(G)$ if and only if they are incident in $G$, see [1].
Vertices $x, y$ of the graph $G$ are said to doubly resolve vertices $u, v$ of $G$ if $d(u, x) - d(u, y) \neq d(v, x) - d(v, y)$.
A set $S$ of vertices of the graph $G$ is a doubly resolving set of $G$ if every two distinct vertices of $G$ are doubly resolved by some two vertices of $S$.  The minimal doubly resolving set of vertices of graph $G$ is a doubly resolving set with minimum cardinality and is denoted by  $\psi(G)$.
The notion of a doubly resolving set  of vertices of the graph $G$ introduced by C\'{a}ceres et al.[2].
A vertex $w$ strongly resolves two vertices $u$ and $v$ if $u$ belongs
to a shortest $v - w$ path or $v$ belongs to a shortest $u - w$ path.
A vertex set $S$ of the graph $G$ is a strong resolving
set of $G$ if every two distinct vertices of $G$ are strongly resolved by some vertex of $S$. A strong metric basis of $G$ is denoted by $sdim(G)$  defined as the minimum cardinality  of a strong resolving set of $G$.
The notion of a strong metric dimension problem set  of vertices of the graph $G$ introduced by A. Seb\"{o} and E. Tannier [3] and further investigated by O. R. Oellermann and  Peters-Fransen [4]. The minimal doubly resolving sets  for  jellyfish   and cocktail party
graphs has been obtained in [5]. For more results related to these concepts see [6-16].
In this paper, first, we construct a class of  graphs of order $2n+ \Sigma_{r=1}^{k-2}nm^{r}$, denoted by $LSG(n,m, k)$, and call these graphs as the layer Sun graphs  with parameters $n$, $m$ and $k$, which is defined as follows:

Let $n, m, k$ be integers such that $n, k\geq 3$, $m\geq 2$ and $G$ be a graph with vertex set $V(G)=V_1 \cup V_2\cup ...\cup V_k$, where   $V_1, V_2, ..., V_k$ are called the layers  of $G$ such that $V_1=V(C_n)=\{1, 2, ..., n\}$,  $V_2=\{v_1, v_2, ..., v_n\}$, and for $l\geq 3$ we have
$V_l=\{B^{(l)}_{1_1}, B^{(l)}_{1_2}, ..., B^{(l)}_{1_{m^{l-3}}}; B^{(l)}_{2_1}, B^{(l)}_{2_2}, ..., B^{(l)}_{2_{m^{l-3}}};  ...; B^{(l)}_{n_1}, B^{(l)}_{n_2}, ..., B^{(l)}_{n_{m^{l-3}}}\}$, and let   $B^{(l)}_{ij}=\{\cup_{t=1} ^ {m} (v_{ij}, t)^{l}\}$, such that every $(v_{ij}, t)^{l}$ is a vertex in the layer $V_l$, and $B^{(l)}_{ij}\cong \overline{K_m}$ in the layer $V_l$,
$1\leq i \leq n$, $1\leq j \leq {m^{l-3}}$, $1\leq t \leq m$, where  $\overline{K_m}$ is the complement of the complete graph on $m$ vertices.  Now, suppose that every vertex $i$ in the cycle $C_n$   is adjacent to exactly  one vertex  in the layer
 $V_2$ say $v_i\in V_2 $, and every vertex $v_i$ in the layer $V_2$ is adjacent to exactly $m$ vertices
$(v_{i1}, 1)^3, (v_{i1}, 2)^3, ..., (v_{i1}, m)^3\in B^{(3)}_{i1}\in V_3$, in particular for $l\geq 3$,  every vertex $(v_{ir}, t)^l \in B^{(l)}_{i_r}\in V_l $ is adjacent to  exactly $m$ vertices $\cup_{t=1} ^ m (v_{ij}, t)^{l+1}\in B^{(l+1)}_{i_j}\in V_{l+1}$, and
then, the resulting graph is called the layer Sun  graph $LSG(n, m, k)$ with parameters $n$, $m$ and $k$.
Also, for $l\geq 3$,  we recall that $B^{(l)}_{i_j}$ as  the  components of the layer $V_l$, $1\leq i \leq n$, $1\leq j \leq {m^{l-3}}$. In particular, we say that two components
$B^{(l)}_{i_j}, B^{(l)}_{r_s}\in V_l$, $1\leq i, r \leq n$, $1\leq j, s \leq m^{l-3}$  are fundamental if $i=r$ and $j\neq s$.
It is natural to consider its vertex set  of layer Sun  graph $LSG(n, m, k)$ as partitioned into $k$ layers. The layers $V_1$ and $V_2$ consist of the vertices $\{1, 2, ..., n\}$ and $\{v_1, v_2, ..., v_n\}$, respectively. In particular, each layer $V_l$ ($l\geq 3$),  consists of the $nm^{l-2}$ vertices. Note that, for each vertex $i$ in the layer $V_1$ and every vertex $x\in B^{(l)}_{ij}\in V_l$,  $l\geq 3$, $1\leq j \leq {m^{l-3}}$, we have
$d(i, x)=l-1$. In this paper, we consider the problem of determining the cardinality $\psi(LSG(n, m, k))$  of minimal doubly resolving sets of  the  layer Sun graph $LSG(n,m, k)$. First, we find the metric dimension of the layer Sun graph $LSG(n, m, k)$; in fact, we prove that  if $n, k\geq3$  and $m\geq2$, then the metric dimension of  layer Sun  graph $LSG(n, m, k)$ is  $nm^{k-2}-nm^{k-3}$.  Moreover, we consider the problem of determining the cardinality $\psi(LSG(n,m, k))$  of minimal doubly resolving sets of $LSG(n,m, k)$ and the strong metric dimension  for layer Sun  graph $LSG(n, m, k)$ and the line graph of the layer Sun  graph $LSG(n, m, k)$. The  graph $LSG(3, 3, 4)$ is shown in Figure 1.
\begin{center}
\begin{tikzpicture}[line cap=round,line join=round,>=triangle 45,x=4.0cm,y=4.0cm]
\clip(4.000504972576445,1.0966067714325458) rectangle (7.500412173920532,3.9311529437730006);
\draw (5.8,2.6)-- (5.4,2.2);
\draw (5.4,2.2)-- (6.2,2.2);
\draw (6.2,2.2)-- (5.8,2.6);
\draw (5.8,2.6)-- (5.8,2.8);
\draw (5.4,2.2)-- (5.2,2.2);
\draw (6.2,2.2)-- (6.4,2.2);
\draw (5.8,2.8)-- (6.,2.8);
\draw (5.8,2.8)-- (5.8,3.);
\draw (5.8,2.8)-- (5.6,2.8);
\draw (5.2,2.2)-- (5.2,2.4);
\draw (5.2,2.2)-- (5.,2.2);
\draw (5.2,2.2)-- (5.2,2.);
\draw (6.4,2.2)-- (6.4,2.4);
\draw (6.4,2.2)-- (6.6,2.2);
\draw (6.4,2.2)-- (6.4,2.);
\draw (6.,2.8)-- (6.4,2.8);
\draw (6.,2.8)-- (6.5,3.);
\draw (6.,2.8)-- (6.2,3.2);
\draw (5.8,3.)-- (6.,3.4);
\draw (5.8,3.)-- (5.6,3.6);
\draw (5.8,3.)-- (5.4,3.4);
\draw (5.6,2.8)-- (5.2,3.2);
\draw (5.6,2.8)-- (5.,3.2);
\draw (5.6,2.8)-- (5.2,2.8);
\draw (5.2,2.4)-- (5.,2.6);
\draw (5.2,2.4)-- (4.8,2.6);
\draw (5.2,2.4)-- (4.8,2.4);
\draw (5.,2.2)-- (4.8,2.2);
\draw (5.,2.2)-- (4.5,2.);
\draw (5.,2.2)-- (5.,2.);
\draw (5.2,2.)-- (5.2,1.8);
\draw (5.2,2.)-- (5.4,1.8);
\draw (5.2,2.)-- (5.6,1.8);
\draw (6.4,2.)-- (6.,1.8);
\draw (6.4,2.)-- (6.2,1.8);
\draw (6.4,2.)-- (6.4,1.8);
\draw (6.6,2.2)-- (6.6,2.);
\draw (6.6,2.2)-- (7.,2.);
\draw (6.6,2.2)-- (6.8,2.2);
\draw (6.4,2.4)-- (6.8,2.4);
\draw (6.4,2.4)-- (6.8,2.6);
\draw (6.4,2.4)-- (6.6,2.6);
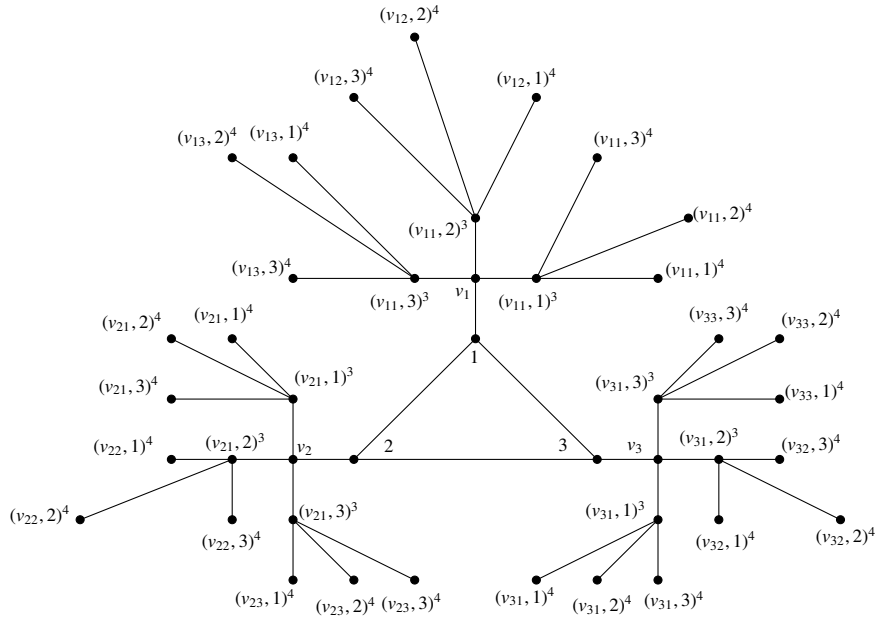
\draw (5.515147965430146,1.4806769589061572) node[anchor=north west] {Figure 1: LSG(3, 3, 4)};
\begin{scriptsize}
\draw [fill=black] (5.8,2.6) circle (1.5pt);
\draw[color=black] (5.7964388069601185,2.540927053903732) node {$1$};
\draw [fill=black] (5.4,2.2) circle (1.5pt);
\draw[color=black] (5.515147965430146,2.245045651633955) node {$2$};
\draw [fill=black] (6.2,2.2) circle (1.5pt);
\draw[color=black] (6.088548527010475,2.249636212373763) node {$3$};
\draw [fill=black] (5.8,2.8) circle (1.5pt);
\draw[color=black] (5.75988777879705,2.751895185051209) node {$v_1$};
\draw [fill=black] (5.2,2.2) circle (1.5pt);
\draw[color=black] (5.239266563160365,2.232092847934913) node {$v_2$};
\draw [fill=black] (6.4,2.2) circle (1.5pt);
\draw[color=black] (6.3249260974181465,2.2404550908941465) node {$v_3$};
\draw [fill=black] (6.,2.8) circle (1.5pt);
\draw[color=black] (5.974950302546447,2.72484798875025) node {$(v_{11}, 1)^3$};
\draw [fill=black] (5.8,3.) circle (1.5pt);
\draw[color=black] (5.679565068414557,2.962863316198685) node {$(v_{11}, 2)^3$};
\draw [fill=black] (5.6,2.8) circle (1.5pt);
\draw[color=black] (5.5530140402514885,2.72485745791017) node {$(v_{11}, 3)^3$};
\draw [fill=black] (5.2,2.4) circle (1.5pt);
\draw[color=black] (5.3041798342826665,2.4738798576027734) node {$(v_{21}, 1)^3$};
\draw [fill=black] (5.,2.2) circle (1.5pt);
\draw[color=black] (5.012070114232309,2.257502287195105) node {$(v_{21}, 2)^3$};
\draw [fill=black] (5.2,2.) circle (1.5pt);
\draw[color=black] (5.3187703950224745,2.0283058382670534) node {$(v_{21}, 3)^3$};
\draw [fill=black] (6.4,2.4) circle (1.5pt);
\draw[color=black] (6.294107218897763,2.4576515398221983) node {$(v_{31}, 3)^3$};
\draw [fill=black] (6.6,2.2) circle (1.5pt);
\draw[color=black] (6.5699886211675445,2.2737306049756803) node {$(v_{31}, 2)^3$};
\draw [fill=black] (6.4,2.) circle (1.5pt);
\draw[color=black] (6.261650583336612,2.031124716787437) node {$(v_{31}, 1)^3$};
\draw [fill=black] (6.4,2.8) circle (1.5pt);
\draw[color=black] (6.524256471524859,2.8330367739540843) node {$(v_{11}, 1)^4$};
\draw [fill=black] (6.5,3.) circle (1.5pt);
\draw[color=black] (6.6145791819073525,3.0169577088006023) node {$(v_{11}, 2)^4$};
\draw [fill=black] (6.2,3.2) circle (1.5pt);
\draw[color=black] (6.288697779637571,3.2657919147694208) node {$(v_{11}, 3)^4$};
\draw [fill=black] (6.,3.4) circle (1.5pt);
\draw[color=black] (5.969540863286256,3.4654449992586228) node {$(v_{12}, 1)^4$};
\draw [fill=black] (5.6,3.6) circle (1.5pt);
\draw[color=black] (5.585470675812639,3.675594251885716) node {$(v_{12}, 2)^4$};
\draw [fill=black] (5.4,3.4) circle (1.5pt);
\draw[color=black] (5.369093105404968,3.4708544385188143) node {$(v_{12}, 3)^4$};
\draw [fill=black] (5.2,3.2) circle (1.5pt);
\draw[color=black] (5.158124974257488,3.2852957466315296) node {$(v_{13}, 1)^4$};
\draw [fill=black] (5.,3.2) circle (1.5pt);
\draw[color=black] (4.920109646809049,3.2690674288509545) node {$(v_{13}, 2)^4$};
\draw [fill=black] (5.2,2.8) circle (1.5pt);
\draw[color=black] (5.098621142395379,2.837131166556001) node {$(v_{13}, 3)^4$};
\draw [fill=black] (5.,2.6) circle (1.5pt);
\draw[color=black] (4.974204039410967,2.682714063571592) node {$(v_{21}, 1)^4$};
\draw [fill=black] (4.8,2.6) circle (1.5pt);
\draw[color=black] (4.671771608978118,2.656163035408525) node {$(v_{21}, 2)^4$};
\draw [fill=black] (4.8,2.4) circle (1.5pt);
\draw[color=black] (4.660952730457735,2.445194904261048) node {$(v_{21}, 3)^4$};
\draw [fill=black] (4.8,2.2) circle (1.5pt);
\draw[color=black] (4.643905534156776,2.245864530154338) node {$(v_{22}, 1)^4$};
\draw [fill=black] (4.5,2.) circle (1.5pt);
\draw[color=black] (4.368843010407378,2.016211445665136) node {$(v_{22}, 2)^4$};
\draw [fill=black] (5.,2.) circle (1.5pt);
\draw[color=black] (4.99928525329433,1.9296604175020688) node {$(v_{22}, 3)^4$};
\draw [fill=black] (5.2,1.8) circle (1.5pt);
\draw[color=black] (5.109440020915763,1.7403300433953592) node {$(v_{23}, 1)^4$};
\draw [fill=black] (5.4,1.8) circle (1.5pt);
\draw[color=black] (5.37450254466516,1.7132828470944006) node {$(v_{23}, 2)^4$};
\draw [fill=black] (5.6,1.8) circle (1.5pt);
\draw[color=black] (5.590880115072831,1.7241017256147841) node {$(v_{23}, 3)^4$};
\draw [fill=black] (6.,1.8) circle (1.5pt);
\draw[color=black] (5.964131424026064,1.7457394826555508) node {$(v_{31}, 1)^4$};
\draw [fill=black] (6.2,1.8) circle (1.5pt);
\draw[color=black] (6.218375069255078,1.7186922863545924) node {$(v_{31}, 2)^4$};
\draw [fill=black] (6.4,1.8) circle (1.5pt);
\draw[color=black] (6.453933761142366,1.7332828470944006) node {$(v_{31}, 3)^4$};
\draw [fill=black] (6.6,2.) circle (1.5pt);
\draw[color=black] (6.624083013769463,1.9296604175020688) node {$(v_{32}, 1)^4$};
\draw [fill=black] (7.,2.) circle (1.5pt);
\draw[color=black] (7.018972079763463,1.9458887352826442) node {$(v_{32}, 2)^4$};
\draw [fill=black] (6.8,2.2) circle (1.5pt);
\draw[color=black] (6.905373855299436,2.252589016072805) node {$(v_{32}, 3)^4$};
\draw [fill=black] (6.8,2.4) circle (1.5pt);
\draw[color=black] (6.916192733819819,2.427328829439706) node {$(v_{33}, 1)^4$};
\draw [fill=black] (6.8,2.6) circle (1.5pt);
\draw[color=black] (6.905373855299436,2.6653441568881413) node {$(v_{33}, 2)^4$};
\draw [fill=black] (6.6,2.6) circle (1.5pt);
\draw[color=black] (6.6024452567286955,2.686485745791017) node {$(v_{33}, 3)^4$};
\end{scriptsize}
\end{tikzpicture}
\end{center}
\section{Definitions And Preliminaries}

\begin{definition} \label{b.3}
Let $G$ be a graph. A vertex $x\in V(G)$ is said to resolve a pair $u, v \in V(G)$ if $d_{G}(u, x)\neq d_{G}(v, x)$. For an ordered subset
$W = \{w_1, w_2, ..., w_k\}$ of vertices in the graph $G$ and a vertex $v$ of $G$, the metric representation of $v$ with respect to $W$ is the $k$-vector $r(v | W) = (d(v, w_1), d(v, w_2), ..., d(v, w_k ))$.
If every pair of distinct vertices of $G$ have different
metric representations, then the ordered set $W$ is called a resolving set of $G$. If the set $W$ is as small
as possible, then it is called a metric basis of the graph $G$. We recall that the  metric dimension of $G$, denoted by
$\beta(G)$, is defined as the minimum cardinality of a resolving set for $G$.
\end{definition}
\begin{proposition} \label{b.3}
Let $G$ be a graph. It is well known that  a doubly resolving set is also a resolving set and
$\beta(G) \leq\psi(G)$. In particular, every
strong resolving set is a resolving set and $\beta(G) \leq sdim(G)$.
\end{proposition}
\section{Main Results}
\noindent
\textbf{Minimal Doubly Resolving Sets and the Strong Metric Dimension  for  Layer Sun  Graph $LSG(n, m, k)$}\\
\newline
\begin{theorem}\label{f.1}
Let $G=LSG(n, m, k)$ be the layer Sun  graph which is defined already. Suppose that $n, m, k$ are integers such that $n, k\geq 3$ and $m\geq 2$.
Then, the metric dimension of  $LSG(n, m, k)$ is  $nm^{k-2}-nm^{k-3}$.
\end{theorem}
\begin{proof}
Let $V(G)=V_1 \cup V_2\cup  ...\cup V_k$, where $V_1, V_2, ..., V_k$,  are called the layers  of  vertices in the layer Sun  graph $LSG(n, m, k)$, which is defined already. It is clear that if $W$  is  an ordered subset  of the layers $V_1 \cup V_2\cup ...\cup V_{k-1}$, then $W$  is not a resolving set in $LSG(n, m, k)$. We may assume that the layer $V_k$ is equal to
$$V_k=\{B^{(k)}_{1_1}, B^{(k)}_{1_2}, ..., B^{(k)}_{1_{m^{k-3}}}; B^{(k)}_{2_1}, B^{(k)}_{2_2}, ..., B^{(k)}_{2_{m^{k-3}}};  ...; B^{(k)}_{n_1}, B^{(k)}_{n_2}, ..., B^{(k)}_{n_{m^{k-3}}}\},$$ where $B^{(k)}_{ij}=\{\cup_{t=1} ^ {m} (v_{ij}, t)^{k}\}$,
$1\leq i \leq n$, $1\leq j \leq {m^{k-3}}$. In the following cases, it can be shown that  the metric dimension of the layer Sun  graph $LSG(n, m, k)$ is  $nm^{k-2}-nm^{k-3}$.\newline

Case 1:  Let $W$  be an ordered subset of the layer $V_k$ in the layer Sun  graph $LSG(n, m, k)$ such that
$$W=\{B^{(k)}_{1_2}, ..., B^{(k)}_{1_{m^{k-3}}}; B^{(k)}_{2_1}, B^{(k)}_{2_2}, ..., B^{(k)}_{2_{m^{k-3}}};  ...; B^{(k)}_{n_1}, B^{(k)}_{n_2}, ..., B^{(k)}_{n_{m^{k-3}}}\}.$$
Hence, $V(G)- W=\{V_1, V_2, ..., V_{k-1}, B^{(k)}_{1_1}\}$. We know that the cardinality of $W$ is $nm^{k-2}-m$, because $|B^{(k)}_{i_j}|=m$,$1\leq i \leq n$, $1\leq j \leq {m^{k-3}}$. Therefore, the metric representation of all the  vertices
$(v_{1_1}, 1)^k, (v_{1_1}, 2)^k, ..., (v_{1_1}, m)^k$ in the component $B^{(k)}_{1_1}$ is the same as $nm^{k-2}-m$-vector with respect to $W$. Thus, $W$ is not a resolving set in $LSG(n, m, k)$. \newline

Case 2. Let $W$  be an ordered subset of the layer $V_k$ in the layer Sun  graph $LSG(n, m, k)$ such that
$$W=\{B^{(k)}_{1_1}-\{(v_{1_1}, 1)^k , (v_{1_1}, 2)^k\}, B^{(k)}_{1_2}, ..., B^{(k)}_{1_{m^{k-3}}}; B^{(k)}_{2_1}, B^{(k)}_{2_2}, ..., B^{(k)}_{2_{m^{k-3}}};  ...; B^{(k)}_{n_1}, B^{(k)}_{n_2}, ..., B^{(k)}_{n_{m^{k-3}}}\}.$$ Hence, $V(G)- W=\{V_1, V_2, ..., V_{k-1}, (v_{1_1}, 1)^k , (v_{1_1}, 2)^k \}$. We know that $|W|= nm^{k-2}-2$. Therefore, the metric representation of two vertices  $(v_{1_1}, 1)^k , (v_{1_1}, 2)^k$ in the component $B^{(k)}_{1_1}$  is the same as $nm^{k-2}-2$-vector with respect to $W$. Thus, $W$ is not a resolving set in $LSG(n, m, k)$. \newline

Case 3. Let $W$  be an ordered subset of the layer $V_k$ in the layer Sun graph $LSG(n, m, k)$ such that
$$W=\{B^{(k)}_{1_1}-(v_{1_1}, 1)^k, B^{(k)}_{1_2}-(v_{1_2}, 1)^k, ..., B^{(k)}_{1_{m^{k-3}}}-(v_{1_{m^{k-3}}}, 1)^k;   ...; B^{(k)}_{n_1}-(v_{n_1}, 1)^k, B^{(k)}_{n_2}-(v_{n_2}, 1)^k, ..., B^{(k)}_{n_{m^{k-3}}}-(v_{n_{m^{k-3}}}, 1)^k\}.$$
Hence, $$V(G)- W=\{V_1, V_2, ..., V_{k-1}, (v_{1_1}, 1)^k, ..., (v_{1_{m^{k-3}}}, 1)^k, ...,(v_{n_1}, 1)^k, ..., (v_{n_{m^{k-3}}}, 1)^k \}.$$ We know that $|W|= nm^{k-2}-nm^{k-3}$. We can show that all the vertices in $V(G)- W$ have different representations with respect to $W$.
Let $u$ be the vertex of the layer $V_1=V(C_n)=\{1,2,..., n\}$. We can assume without loss of generality that $u = i$, $1\leq i\leq n$. Hence,
$d(u, (v_{i_j}, t)^k)=k-1 $, where  $(v_{i_j}, t)^k\in B^{(k)}_{i_j}$, $1\leq  t\leq m$, $1\leq j\leq m^{k-3}$; otherwise, if $u \neq i$, then
$d(u, (v_{i_j}, t)^k)>k-1 $. Now, let $u\in V_2=\{v_1, v_2, ..., v_n\}$. We can assume without loss of generality that
$u = v_i$, $1\leq i\leq n$. Hence, $d(u, (v_{i_j}, t)^k)=k-2 $, where  $(v_{i_j}, t)^k\in B^{(k)}_{i_j}$, $1\leq  t\leq m$, $1\leq j\leq m^{k-3}$; otherwise, if $u \neq v_i$, then $d(u, (v_{i_j}, t)^k)>k-2 $.
In a similar way, we can show that all the vertices in the layers $V_3, ..., V_{k-1}$ have different representations with respect to $W$. In particular,  for every vertex $u\in \{(v_{1_1}, 1)^k, ..., (v_{1_{m^{k-3}}}, 1)^k, ..., (v_{n_1}, 1)^k, ..., (v_{n_{m^{k-3}}}, 1)^k\}$,  we have
$d( u , (v_{i_j}, t)^k)=2 $, $2\leq  t\leq m$, if $u=(v_{i_j}, 1)^k$; otherwise, if
$u\in \{(v_{1_1}, 1)^k, ..., (v_{1_{m^{k-3}}}, 1)^k, ..., (v_{n_1}, 1)^k, ..., (v_{n_{m^{k-3}}}, 1)^k\}$ and $u \neq (v_{i_j}, 1)^k$, then $d(u, (v_{i_j}, t)^k)>4 $.
Therefore,  all the vertices in $V(G)- W$ have different representations with respect to $W$. This implies that $W$ is a resolving set in
$LSG(n,m, k)$. From the above-mentioned  cases we can be concluded that the minimum possible cardinality of a resolving set in $LSG(n, m, k)$ is $nm^{k-2}-nm^{k-3}$.
\end{proof}
\begin{theorem}\label{f.2}
Let $G=LSG(n, m, k)$ be the layer Sun  graph which is defined already. Suppose that $n, m, k$ are integers such that $n, k\geq 3$ and $m\geq 2$. Then, the cardinality of minimum doubly resolving set of the $LSG(n, m, k)$ is  $nm^{k-2}$.
\end{theorem}
\begin{proof}
In the following cases, it  can be shown that the cardinality of minimum doubly resolving set of the layer Sun  graph $LSG(n, m, k)$ is  $nm^{k-2}$.\newline

Case 1.
We know that the ordered subset
$$W=\{B^{(k)}_{1_1}-(v_{1_1}, 1)^k, B^{(k)}_{1_2}-(v_{1_2}, 1)^k, ..., B^{(k)}_{1_{m^{k-3}}}-(v_{1_{m^{k-3}}}, 1)^k;   ...; B^{(k)}_{n_1}-(v_{n_1}, 1)^k, B^{(k)}_{n_2}-(v_{n_2}, 1)^k, ..., B^{(k)}_{n_{m^{k-3}}}-(v_{n_{m^{k-3}}}, 1)^k\},$$
of vertices in  the layer $V_k$ of $LSG(n, m, k)$ is   a resolving set for $LSG(n,m, k)$ of  cardinality $nm^{k-2}-nm^{k-3}$. We show  that this subset
is not a doubly resolving set for $LSG(n, m, k)$. Because if the vertex $u=(v_{i_j}, 1)^k\in B^{(k)}_{i_j}\in V_k$, $1\leq i\leq n$, $1\leq j\leq m^{k-3}$ is adjacent to a vertex $v\in V_{k-1}$, then for every $x,y\in W$, we have $d(u, x) - d(u, y) = d(v, x) - d(v, y)$.\newline

Case 2.
Now, let the subset of vertices in   $LSG(n, m, k)$ be
$W=\{B^{(k)}_{1_1}-(v_{1_1}, 1)^k, B^{(k)}_{1_2}, ..., B^{(k)}_{1_{m^{k-3}}},   ..., B^{(k)}_{n_1}, B^{(k)}_{n_2}, ..., B^{(k)}_{n_{m^{k-3}}}\}.$
In a similar fashion as in  Case 3 of  Theorem \ref{f.1}, we can show that all the vertices in the layers $V_1, V_2, ..., V_{k-1}$  of $LSG(n, m, k)$ and the vertex
$(v_{1_1}, 1)^k$ in  the layer $V_k$ of $LSG(n, m, k)$ have different representations with respect to $W$. So, $W$ is a  resolving set in $LSG(n, m, k)$ of  cardinality $nm^{k-2}-1$. Note that, in this case, by  a similar way as in Case 1, we can show  that this subset  is not a doubly resolving set for $LSG(n,m, k)$.\newline

Case 3. Finally, let the subset of vertices in   $LSG(n,m, k)$ be
$W=\{B^{(k)}_{1_1}, B^{(k)}_{1_2}, ..., B^{(k)}_{1_{m^{k-3}}};   ...; B^{(k)}_{n_1}, B^{(k)}_{n_2}, ..., B^{(k)}_{n_{m^{k-3}}}\}.$
In a similar fashion as in Theorem \ref{f.1}, we can show that all the vertices in the layers $V_1, V_2, ..., V_{k-1}$ of $LSG(n, m, k)$ have different representations with respect to $W$. So, this subset is also a  resolving set in $LSG(n, m, k)$ of cardinality  $nm^{k-2}$. We show  that this subset is  a doubly resolving set for $LSG(n, m, k)$. It is  sufficient  to prove that  for two vertices $u$ and $v$ in $LSG(n, m, k)$, there are vertices $x, y \in W$ such that $d(u, x) - d(u, y) \neq d(v, x) - d(v, y)$. Consider two vertices $u$ and $v$ in  $LSG(n, m, k)$. Then, we have the following: \newline

Case 3.1. Suppose that both vertices  $u$ and $v$ lie in the layer $V_1$. Hence,  there are  $r, s \in \{1, 2, ..., n\}$ such that $u=r$ and $v=s$. Moreover, we know that the layer Sun  graph $LSG(n, m, k)$ has the property that, for each vertex $r$ in  the layer $V_1$, there is some  vertex such as $x=(v_{r_j}, t)^k$  in the component $B^{(k)}_{r_j}$,  $1\leq j\leq m^{k-3}$   in  the layer $V_k$,  at distance $k-1$ from $u$; in fact, $d(u, x)=k-1$.
In the same way, there is some  vertex such as $y=(v_{s_j}, t)^k$ in  the component $B^{(k)}_{s_j}$ in  the layer $V_k$,  at distance $k-1$ from $v$.  In particular, it is easy to prove that $d(u, x) - d(u, y)<0$ because $d(u, y)\geq k$. Also, $d(v, x) - d(v, y)>0$  because  $d(v, x)\geq k$.
\newline

Case 3.2. Now, suppose that both vertices  $u$ and $v$ lie in the layer $V_2$. In a similar way as in Case 3.1, we can show that there are vertices $x, y \in W$ such that $d(u, x) - d(u, y) \neq d(v, x) - d(v, y)$.\newline

Case 3.3. Suppose that both vertices  $u$ and $v$ lie in the layer $V_l$, $l\geq 3$ such that these vertices  lie in the  one component of the layer
$V_l$, say $B^{(l)}_{i_j}$, $1\leq i \leq n$, $1\leq j\leq m^{l-3}$.
In this case,  $d(u, v)=2$. Moreover, we know that the layer Sun  graph $LSG(n, m, k)$ has the property that, for each vertex
$u\in B^{(l)}_{i_j}$ in the layer $V_l$, there is a component of the layer $V_k$, say  $B^{(k)}_{i_r}$, $1\leq i\leq n$, $1\leq r \leq m^{k-3}$ such that for any vertex $x\in B^{(k)}_{i_r}$, we have $d(u,x)=k-l$. In the same way, for the vertex $v\in B^{(l)}_{i_j}$ in  the layer $V_l$, there is a component of the layer $V_k$, say  $B^{(k)}_{i_s}$, $1\leq i\leq n$, $1\leq s \leq m^{k-3}$, $r\neq s$ such that, for any vertex $y\in B^{(k)}_{i_s}$, we have $d(v,y)=k-l$. Thus, $d(u, x) - d(u, y) \neq d(v, x) - d(v, y)$ because $d(u,y)=k-l+2$ and $d(v,x)=k-l+2$.\newline

Case 3.4. Suppose that both vertices  $u$ and $v$ lie in the layer $V_l$, $l\geq 3$ such that these vertices   lie in the  two distinct  components of the layer $V_l$. We can assume without loss of generality that $u\in B^{(l)}_{p_{j_1}}$ and
$v\in B^{(l)}_{q_{j_2}}$, $1\leq p, q \leq n$, and $1\leq {j_1}, {j_2}\leq m^{l-3}$.
Moreover, we know that the layer Sun  graph $LSG(n, m, k)$ has the property that, for each vertex $u\in B^{(l)}_{p_{j_1}}$ in  the layer $V_l$, there is a component of the layer $V_k$, say  $B^{(k)}_{p_r}$, $1\leq r \leq m^{k-3}$ such that for any vertex $x\in B^{(k)}_{p_r}$, we have $d(u,x)=k-l$. In the same way, for the vertex $v\in B^{(l)}_{q_{j_2}}$ in  the layer $V_l$, there is a component of the layer $V_k$, say  $B^{(k)}_{q_s}$, $1\leq s \leq m^{k-3}$,  such that for any vertex $y\in B^{(k)}_{q_s}$, we have $d(v,y)=k-l$. In the following,
let two components $ B^{(l)}_{p_{j_1}}$, and $B^{(l)}_{q_{j_2}}$ are fundamental; indeed,  $p=q$. Hence, $d(u, v)=2l-4$, $d(u, y)=d(v, x)=k+l-4$.
Thus, $d(u, x) - d(u, y) \neq d(v, x) - d(v, y)$.
Now, let two components $ B^{(l)}_{p_{j_1}}$, and $B^{(l)}_{q_{j_2}}$ are not fundamental; indeed,  $p\neq q$. Hence $d(u, v)=2l-2+d_{C_n}(p, q)$,
$d(u, y)=d(v, x)=k+l-2+d_{C_n}(p, q)$.
Thus, $d(u, x) - d(u, y) \neq d(v, x) - d(v, y)$.\newline

Case 3.5.
Suppose that  vertices  $u$ and $v$ lie in distinct layers $V_a, V_b$, respectively. Note that if $a=1$ and $b=2$, $a=1$ and $b>2$, or $a=2$ and $b>2$, there is nothing  to do. Now, let $3\leq a< b$. Hence, there is a  component of the layer $V_a$, say
$B^{(a)}_{i_j}$, $1\leq i\leq n$, $1\leq j \leq m^{a-3}$ such that $u\in B^{(a)}_{i_j}$.
Also, there is a  component of the layer $V_b$, say  $B^{(b)}_{p_q}$, $1\leq p\leq n$, $1\leq q \leq m^{b-3}$ such that $v\in B^{(b)}_{p_q}$.
In particular, there is a component of the layer $V_k$, say  $B^{(k)}_{i_r}$, $1\leq i\leq n$, $1\leq r \leq m^{k-3}$ such that for any vertex
$x\in B^{(k)}_{i_r}$ we have $d(u, x)=k-a$. Now, let $i=p$;  if we consider
$y\in B^{(k)}_{z_s}$, $z\neq i$, $1\leq z\leq n$, and $1\leq s \leq m^{k-3}$, then we have $d(u, x) - d(u, y) \neq d(v, x) - d(v, y)$.
Because $d(u, y)=k+a-2+ d_{C_n}(i, z)$, $d(v, y)=k+b-2+ d_{C_n}(i, z)$ and $d(u, x) \neq d(v, x)$.
Note that if $i\neq p$, then  there is a component of the layer $V_k$, say  $B^{(k)}_{p_s}$, $1\leq p\leq n$, $1\leq s \leq m^{k-3}$ such that for any vertex $y\in B^{(k)}_{p_j}$, we have $d(v, y)=k-b$, and then we have $d(u, x) - d(u, y) \neq d(v, x) - d(v, y)$ because $d(u, x) - d(u, y)< 0$ and
$d(v, x) - d(v, y)> 0$.
Thus, from the  above-mentioned  cases we can be concluded that the  cardinality of minimum doubly resolving set of
the layer Sun graph $LSG(n, m, k)$ is  $nm^{k-2}$.
\end{proof}
\begin{theorem}\label{f.3}
Let $G=LSG(n, m, k)$ be the layer Sun  graph which is defined already. Suppose that $n, m, k$ are integers such that $n, k\geq 3$ and $m\geq 2$. Then the strong metric dimension of
$LSG(n, m, k)$ is  $nm^{k-2}-1$.
\end{theorem}
\begin{proof}
In the following cases, it  can be seen that the cardinality of minimum strong resolving set of the layer Sun  graph $LSG(n, m, k)$ is  $nm^{k-2}-1$.\newline

Case 1.
We know that the ordered subset
$$W=\{B^{(k)}_{1_1}-(v_{1_1}, 1)^k, B^{(k)}_{1_2}-(v_{1_2}, 1)^k, ..., B^{(k)}_{1_{m^{k-3}}}-(v_{1_{m^{k-3}}}, 1)^k;   ...; B^{(k)}_{n_1}-(v_{n_1}, 1)^k, B^{(k)}_{n_2}-(v_{n_2}, 1)^k, ..., B^{(k)}_{n_{m^{k-3}}}-(v_{n_{m^{k-3}}}, 1)^k\},$$
of vertices in  the layer $V_k$ of the layer Sun graph $LSG(n, m, k)$ is   a resolving set for $LSG(n, m, k)$ of  cardinality $nm^{k-2}-nm^{k-3}$.
Now, let
$N=V_k- W=\{(v_{1_1}, 1)^k, ..., (v_{1_{m^{k-3}}}, 1)^k, ...,(v_{n_1}, 1)^k, ..., (v_{n_{m^{k-3}}}, 1)^k \}.$
By considering  distinct vertices $u, v\in N$  we can show that there is not a vertex $w\in W$ such that $u$ belongs to a shortest $v - w$ path or $v$ belongs to a shortest $u - w$ path because the valency of every  vertex in the layer $V_k$ is one.
So, this subset  is not a  strong resolving set for $G$.  Thus, we can be conclude that if   $W$ is a strong resolving set for graph $G$,
then $|W|\geq nm^{k-2}-1$, because $|N|$ must be less than $2$.\newline

Case 2.
On the other hand,  we  can show that the subset
$$W=\{B^{(k)}_{1_1}-(v_{1_1}, 1)^k, B^{(k)}_{1_2}, ..., B^{(k)}_{1_{m^{k-3}}},   ..., B^{(k)}_{n_1}, B^{(k)}_{n_2}, ..., B^{(k)}_{n_{m^{k-3}}}\},$$
of vertices in  the graph $G$ is a  resolving set for  graph $G$. We show that  this  subset  is a strong resolving set in  graph $G$.
It is sufficient to prove that  every two distinct vertices $u,v \in V(G)-W$ is strongly resolved by a vertex $w\in W$. Then, we have the following:\newline

Case 2.1. Suppose that both vertices  $u$ and $v$ lie in the layer $V_1$. Hence,  there are  $r, s \in \{1, 2, ..., n\}$ such that $u=r$ and $v=s$. Moreover, we know that the layer Sun  graph $LSG(n, m, k)$ has the property that, for each vertex $r$ in  the layer $V_1$, there is a component  $B^{(k)}_{r_j}$,  $1\leq j\leq m^{k-3}$   in  the layer $V_k$ such that, for every vertex such as $w \in B^{(k)}_{r_j}$, we have $d(u, w)=k-1$ and $d(v, w)>k-1$, and hence, $u$ belongs to a shortest $w - v$ path.\newline

Case 2.2. Now suppose that both vertices  $u$ and $v$ lie in the layer $V_2$. In a similar way as in Case 2.1, we can show that the vertices $u$ and $v$ is strongly resolved by a vertex $w\in W$.
\newline

Case 2.3. Suppose that both vertices  $u$ and $v$ lie in the layer $V_l$, $l\geq 3$ such that these vertices  lie in the  one component of the layer
$V_l$, say $B^{(l)}_{i_j}$, $1\leq i \leq n$, $1\leq j\leq m^{l-3}$.
In this case,  $d(u, v)=2$. Moreover, we know that the layer Sun  graph $LSG(n,m, k)$ has the property that, for each vertex
$u\in B^{(l)}_{i_j}$ in the layer $V_l$, there is a component of the layer $V_k$, say  $B^{(k)}_{i_r}$, $1\leq i\leq n$, $1\leq r \leq m^{k-3}$ such that, for any vertex $w\in B^{(k)}_{i_r}$, we have $d(u,w)=k-l$, and hence, $u$ belongs to a shortest $w - v$ path.\newline

Case 2.4. Suppose that both vertices  $u$ and $v$ lie in the layer $V_l$, $l\geq 3$ such that these vertices   lie in the  two distinct  components of the layer $V_l$. We can assume without loss of generality that $u\in B^{(l)}_{p_{j_1}}$ and
$v\in B^{(l)}_{q_{j_2}}$, $1\leq p, q \leq n$, and $1\leq {j_1}, {j_2}\leq m^{l-3}$.
Moreover, we know that the layer Sun  graph $LSG(n,m, k)$ has the property that, for each vertex $u\in B^{(l)}_{p_{j_1}}$ in  the layer $V_l$, there is a component of the layer $V_k$, say  $B^{(k)}_{p_r}$, $1\leq r \leq m^{k-3}$ such that for any vertex $w\in B^{(k)}_{p_r}$, we have $d(u,w)=k-l$,
and hence, $u$ belongs to a shortest $w - v$ path.\newline

Case 2.5.
Suppose that  vertices  $u$ and $v$ lie in distinct layers $V_a, V_b$, respectively. Note that if $a=1$ and $b=2$, $a=1$ and $b>2$ or $a=2$ and $b>2$ there is nothing  to do. Now, let $3\leq b< a$. Hence, there is a  component of the layer $V_a$, say
$B^{(a)}_{i_j}$, $1\leq i\leq n$, $1\leq j \leq m^{a-3}$ such that $u\in B^{(a)}_{i_j}$.
Also, there is a  component of the layer $V_b$, say  $B^{(b)}_{p_q}$, $1\leq p\leq n$, $1\leq q \leq m^{b-3}$ such that $v\in B^{(b)}_{p_q}$.
In particular, there is a component of the layer $V_k$, say  $B^{(k)}_{i_r}$, $1\leq i\leq n$, $1\leq r \leq m^{k-3}$ such that, for any vertex
$w\in B^{(k)}_{i_r}$, we have $d(u, w)=k-a$, and hence, $u$ belongs to a shortest $w - v$ path.\newline

Case 2.6.
Let  $u, v$ be two distinct vertices in $V(G)-W$ such that  $u=(v_{1_1}, 1)^k\in V_k$ and $ v\in V_l$, $l\geq 3$. So, there is a component $B^{(l)}_{r_{j_1}}$, in the layer $V_{l}$, $1\leq r \leq n$ and
$1\leq {j_1} \leq m^{l-3}$ such that $v\in B^{(l)}_{r_{j_1}}$. Thus,  there is some vertex in the component $B^{(k)}_{r_{j_2}}$, $1\leq {j_2} \leq m^{k-3}$ say $w$ such that $d(w, v)=k-l$, $d(w, u)>k-l$, and $v$ belongs to a shortest $w - u$ path.\newline

Case 2.7.
Let $u, v$ be two distinct vertices in $V(G)-W$ such that  $u=(v_{1_1}, 1)^k\in V_k$ and
$ v\in V_1$.  So, there is some   $i\in V_1= \{1, 2, ..., n \}$
such that $v=i$. If $i=1$, indeed  $d(u, v)= k-1$, and then, there is a component $B^{(k)}_{r_j}$, in the layer $V_k$, $r\neq 1$ and $1\leq j \leq m^{k-3}$ such that, for every vertex such as $w$ in the component $B^{(k)}_{r_j}$ we have $d(w, v)\geq k$,  $d(w, u)\geq 2k-1$ and $v$ belongs to a shortest $w - u$ path.
Now, let $i\neq 1$; indeed  $d(u, v)\geq k$, and hence, there  is a component $B^{(k)}_{i_j}$, in the layer $V_k$,  $1\leq j \leq m^{k-3}$ such that,
for every vertices such as $w$ in the components $B^{(k)}_{i_j}$, we have $d(w, v)=k-1$, $d(w, u)\geq 2k-1$, and $v$ belongs to a shortest $w - u$ path.
\newline
Thus, from the above-mentioned  cases we can be concluded that the  cardinality of minimum strong resolving set of
the layer Sun graph $LSG(n, m, k)$ is  $nm^{k-2}-1$.
\end{proof}
\noindent \textbf{ Minimal Doubly Resolving Sets and the Strong Metric Dimension  for the Line Graph of  Layer Sun Graph $LSG(n, m, k)$ }\\
Let $G=LSG(n, m, k)$ be the layer Sun  graph which is defined already. Now,
let $H$ be a graph with vertex set $V(H)=U_1 \cup U_2\cup ...\cup U_k$, where   $U_1, U_2, ..., U_k$ are called the layers  of $H$ which is defined as follows:

Let $U_1=V(C_n)=\{1, 2, ..., n\}$,  $U_2=\{u_1, u_2, ..., u_n\}$, and for $l\geq 3$ we have
$$U_l=\{D^{(l)}_{1_1}, D^{(l)}_{1_2}, ..., D^{(l)}_{1_{m^{l-3}}}; D^{(l)}_{2_1}, D^{(l)}_{2_2}, ..., D^{(l)}_{2_{m^{l-3}}};  ...; D^{(l)}_{n_1}, D^{(l)}_{n_2}, ..., D^{(l)}_{n_{m^{l-3}}}\},$$ and let   $D^{(l)}_{ij}=\{\cup_{t=1} ^ {m} (u_{ij}, t)^{l}\}$ such that  every $(u_{ij}, t)^{l}$ is a vertex in the layer $U_l$ and $D^{(l)}_{ij}\cong K_m$ in the layer $U_l$,
$1\leq i \leq n$, $1\leq j \leq {m^{l-3}}$, and $1\leq t \leq m$, where  $K_m$ is the complete graph on $m$ vertices.
Now, suppose that every vertex $i\in\{2, 3, ..., n\}$ in the cycle  $C_n$  or the layer $U_1$  is adjacent to exactly  two vertices  in the layer
$U_2$, say $u_i, u_{i-1} \in U_2 $. In particular, for the vertex $1$  in the layer $U_1$, we have $1$  is adjacent to exactly  two vertices  in the layer $U_2$, say $u_1, u_{n} \in U_2$.
 Also,  every vertex $u_i$ in the layer $U_2$ is adjacent to exactly $m$ vertices
$(u_{i1}, 1)^3, (u_{i1}, 2)^3, ..., (u_{i1}, m)^3\in D^{(3)}_{i1}\in U_3$, in particular for $l\geq 3$,  every vertex $(u_{ir}, t)^l \in D^{(l)}_{i_r}\in U_l $ is adjacent to  exactly $m$ vertices $\cup_{t=1} ^ m (u_{ij}, t)^{l+1}\in D^{(l+1)}_{i_j}\in U_{l+1}$, and
then the resulting graph is isomorphic with  the line graph of the layer Sun graph $LSG(n, m, k)$ with parameters $n$, $m$ and $k$;
in fact, $L(G)\cong H$. Note that  simply we use refinement of the natural relabelling of the line graph
of the graph $LSG(n, m, k)$.
Also, for $l\geq 3$,  we recall that $D^{(l)}_{i_j}$ as the  components of $U_l$, $1\leq i \leq n$, $1\leq j \leq {m^{l-3}}$. In particular, we say that two components
$D^{(l)}_{i_j}, D^{(l)}_{r_s}$,$1\leq i, r \leq n$, $1\leq j, s \leq m^{l-3}$  are fundamental if $i=r$ and $j\neq s$.
It is natural to consider its vertex set  of the line graph of the layer Sun  graph $LSG(n, m, k)$ is also as partitioned into $k$ layers. The layers $U_1$ and $U_2$ consist of the vertices $\{1, 2, ..., n\}$ and $\{u_1, u_2, ..., u_n\}$, respectively. In particular, each layer $U_l$ ($l\geq 3$),  consists of the $nm^{l-2}$ vertices. Note that, for each vertex $i$ in the layer $U_1$, and every vertex $x\in D^{(l)}_{ij}\in U_l$,  $l\geq 3$, $1\leq j \leq {m^{l-3}}$, we have
$d(i, x)=l-1$. In this section, we consider the problem of determining the cardinality $\psi(L(G))$  of minimal doubly resolving sets of  the line graph of the layer Sun  graph $LSG(n, m, k)$. We find the minimal doubly resolving set for the line graph of the layer Sun  graph $LSG(n, m, k)$, and in fact, we prove that  if $n, k\geq3$  and $m\geq2$ then the  minimal doubly resolving set of  the line graph of the layer Sun  graph $LSG(n, m, k)$ is $nm^{k-2}-nm^{k-3}$.
 Figure 2 shows the line graph of the graph $LSG(3, 3, 4)$. Note that  simply we use refinement of the natural relabelling of the line graph
of the graph $LSG(3, 3, 4)$.

\begin{center}
\begin{tikzpicture}[line cap=round,line join=round,>=triangle 45,x=4.0cm,y=4.0cm]
\clip(6.622968924336641,0.40745613982090334) rectangle (10.200267793037511,3.1914421740027);
\draw (8.4,2.4)-- (8.,2.);
\draw (8.,2.)-- (8.8,2.);
\draw (8.8,2.)-- (8.4,2.4);
\draw (8.8,2.4)-- (8.4,2.4);
\draw (8.8,2.4)-- (8.8,2.);
\draw (8.,2.4)-- (8.4,2.4);
\draw (8.,2.4)-- (8.,2.);
\draw (8.,2.)-- (8.4,1.8);
\draw (8.4,1.8)-- (8.8,2.);
\draw (7.8,2.6)-- (7.6,2.6);
\draw (7.6,2.6)-- (7.6,2.4);
\draw (7.6,2.4)-- (7.8,2.6);
\draw (7.8,2.6)-- (8.,2.4);
\draw (7.6,2.6)-- (8.,2.4);
\draw (7.6,2.4)-- (8.,2.4);
\draw (8.2,1.6)-- (8.4,1.4);
\draw (8.4,1.4)-- (8.6,1.6);
\draw (8.6,1.6)-- (8.2,1.6);
\draw (8.2,1.6)-- (8.4,1.8);
\draw (8.4,1.4)-- (8.4,1.8);
\draw (8.6,1.6)-- (8.4,1.8);
\draw (9.2,2.4)-- (9.2,2.6);
\draw (9.2,2.6)-- (9.,2.6);
\draw (9.2,2.4)-- (9.,2.6);
\draw (9.2,2.4)-- (8.8,2.4);
\draw (9.2,2.6)-- (8.8,2.4);
\draw (9.,2.6)-- (8.8,2.4);
\draw (8.2,2.6)-- (8.2,2.8);
\draw (8.2,2.8)-- (7.8,2.8);
\draw (7.8,2.8)-- (8.2,2.6);
\draw (8.2,2.6)-- (7.8,2.6);
\draw (8.2,2.8)-- (7.8,2.6);
\draw (7.8,2.8)-- (7.8,2.6);
\draw (7.4,2.8)-- (7.,2.8);
\draw (7.,2.8)-- (7.,2.6);
\draw (7.,2.6)-- (7.4,2.8);
\draw (7.4,2.8)-- (7.6,2.6);
\draw (7.,2.8)-- (7.6,2.6);
\draw (7.,2.6)-- (7.6,2.6);
\draw (7.,2.4)-- (7.,2.2);
\draw (7.,2.2)-- (7.4,2.2);
\draw (7.4,2.2)-- (7.,2.4);
\draw (7.,2.4)-- (7.6,2.4);
\draw (7.4,2.2)-- (7.6,2.4);
\draw (7.4,1.6)-- (7.4,1.4);
\draw (7.4,1.4)-- (7.8,1.4);
\draw (7.4,1.6)-- (7.8,1.4);
\draw (7.4,1.6)-- (8.2,1.6);
\draw (7.8,1.4)-- (8.2,1.6);
\draw (7.4,1.4)-- (8.2,1.6);
\draw (8.,1.2)-- (8.4,1.);
\draw (8.8,1.2)-- (8.4,1.);
\draw (8.8,1.2)-- (8.,1.2);
\draw (8.,1.2)-- (8.4,1.4);
\draw (8.4,1.)-- (8.4,1.4);
\draw (8.8,1.2)-- (8.4,1.4);
\draw (9.,1.4)-- (9.4,1.4);
\draw (9.4,1.4)-- (9.4,1.6);
\draw (9.4,1.6)-- (9.,1.4);
\draw (9.,1.4)-- (8.6,1.6);
\draw (9.4,1.4)-- (8.6,1.6);
\draw (9.4,1.6)-- (8.6,1.6);
\draw (9.4,2.2)-- (9.8,2.2);
\draw (9.8,2.2)-- (9.8,2.4);
\draw (9.8,2.4)-- (9.4,2.2);
\draw (9.8,2.2)-- (9.2,2.4);
\draw (9.4,2.2)-- (9.2,2.4);
\draw (9.8,2.4)-- (9.2,2.4);
\draw (9.8,2.6)-- (9.8,2.8);
\draw (9.8,2.8)-- (9.4,2.8);
\draw (9.4,2.8)-- (9.8,2.6);
\draw (9.8,2.6)-- (9.2,2.6);
\draw (9.8,2.8)-- (9.2,2.6);
\draw (9.4,2.8)-- (9.2,2.6);
\draw (9.,2.8)-- (8.6,2.8);
\draw (8.6,2.6)-- (8.6,2.8);
\draw (8.6,2.6)-- (9.,2.8);
\draw (9.,2.8)-- (8.6,2.8);
\draw (9.,2.8)-- (9.,2.6);
\draw (8.6,2.8)-- (9.,2.6);
\draw (8.6,2.6)-- (9.,2.6);
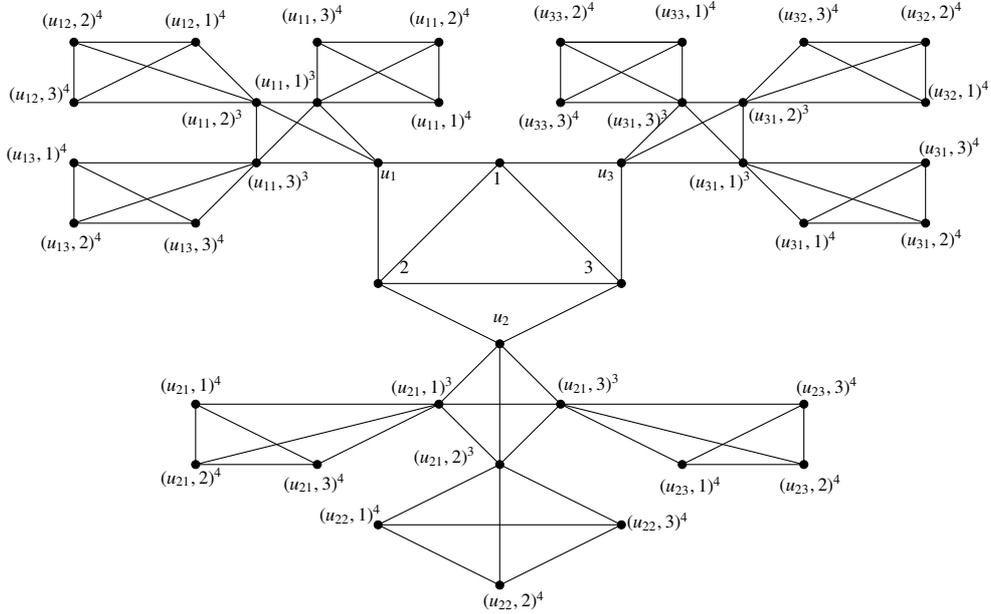
\draw (8.017425645882645,0.7868666258952373) node[anchor=north west] {Figure 2: The line graph $LSG(3, 3, 4)$};
\draw (7.,2.2)-- (7.6,2.4);
\begin{scriptsize}
\draw [fill=black] (8.4,2.4) circle (1.5pt);
\draw[color=black] (8.39190872304685,2.3488552503830764) node {$1$};
\draw [fill=black] (8.,2.) circle (1.5pt);
\draw[color=black] (8.086409370623361,2.0532107157797) node {$2$};
\draw [fill=black] (8.8,2.) circle (1.5pt);
\draw[color=black] (8.692480666560284,2.0532107157797) node {$3$};
\draw [fill=black] (8.,2.4) circle (1.5pt);
\draw[color=black] (8.037135281522797,2.3585648860233015) node {$u_1$};
\draw [fill=black] (8.4,1.8) circle (1.5pt);
\draw[color=black] (8.406690949777019,1.8807514039277302) node {$u_2$};
\draw [fill=black] (8.8,2.4) circle (1.5pt);
\draw[color=black] (8.75160957348096,2.358710068203189) node {$u_3$};
\draw [fill=black] (7.8,2.6) circle (1.5pt);
\draw[color=black] (7.696853702369138,2.6740642384467908) node {$(u_{11}, 1)^3$};
\draw [fill=black] (7.6,2.6) circle (1.5pt);
\draw[color=black] (7.455701030136155,2.545661242425553) node {$(u_{11}, 2)^3$};
\draw [fill=black] (7.6,2.4) circle (1.5pt);
\draw[color=black] (7.673523297437843,2.339000432562964) node {$(u_{11}, 3)^3$};
\draw [fill=black] (8.2,1.6) circle (1.5pt);
\draw[color=black] (8.145538277544036,1.6590180029751977) node {$(u_{21}, 1)^3$};
\draw [fill=black] (8.4,1.4) circle (1.5pt);
\draw[color=black] (8.218723500295444,1.4274297842025527) node {$(u_{21}, 2)^3$};
\draw [fill=black] (8.6,1.6) circle (1.5pt);
\draw[color=black] (8.693061395279835,1.6688728207953103) node {$(u_{21}, 3)^3$};
\draw [fill=black] (9.2,2.4) circle (1.5pt);
\draw[color=black] (9.125366739745801,2.339000432562964) node {$(u_{31}, 1)^3$};
\draw [fill=black] (9.2,2.6) circle (1.5pt);
\draw[color=black] (9.319423055576534,2.5607338335154965) node {$(u_{31}, 2)^3$};
\draw [fill=black] (9.,2.6) circle (1.5pt);
\draw[color=black] (8.854794796232367,2.545370878065778) node {$(u_{31}, 3)^3$};
\draw [fill=black] (8.2,2.6) circle (1.5pt);
\draw[color=black] (8.209594593374767,2.5410241978752715) node {$(u_{11}, 1)^4$};
\draw [fill=black] (8.2,2.8) circle (1.5pt);
\draw[color=black] (8.20466718446471,2.8810154126691545) node {$(u_{11}, 2)^4$};
\draw [fill=black] (7.8,2.8) circle (1.5pt);
\draw[color=black] (7.785837427109927,2.890870230489267) node {$(u_{11}, 3)^4$};
\draw [fill=black] (7.4,2.8) circle (1.5pt);
\draw[color=black] (7.401499532125536,2.876088003759098) node {$(u_{12}, 1)^4$};
\draw [fill=black] (7.,2.8) circle (1.5pt);
\draw[color=black] (6.997452001500919,2.876088003759098) node {$(u_{12}, 2)^4$};
\draw [fill=black] (7.,2.6) circle (1.5pt);
\draw[color=black] (6.889049005479681,2.629717558256284) node {$(u_{12}, 3)^4$};
\draw [fill=black] (7.,2.4) circle (1.5pt);
\draw[color=black] (6.879194187659569,2.432621201854033) node {$(u_{13}, 1)^4$};
\draw [fill=black] (7.,2.2) circle (1.5pt);
\draw[color=black] (6.992524592590863,2.1369766672506567) node {$(u_{13}, 2)^4$};
\draw [fill=black] (7.4,2.2) circle (1.5pt);
\draw[color=black] (7.3965721232154795,2.1320492583406003) node {$(u_{13}, 3)^4$};
\draw [fill=black] (7.4,1.6) circle (1.5pt);
\draw[color=black] (7.386717305395366,1.6633646831657043) node {$(u_{21}, 1)^4$};
\draw [fill=black] (7.4,1.4) circle (1.5pt);
\draw[color=black] (7.386717305395366,1.3584460594617649) node {$(u_{21}, 2)^4$};
\draw [fill=black] (7.8,1.4) circle (1.5pt);
\draw[color=black] (7.790764836019982,1.3387364238215398) node {$(u_{21}, 3)^4$};
\draw [fill=black] (8.,1.2) circle (1.5pt);
\draw[color=black] (7.909022649861333,1.235260836710358) node {$(u_{22}, 1)^4$};
\draw [fill=black] (8.4,1.) circle (1.5pt);
\draw[color=black] (8.446690949777019,0.9494711199270938) node {$(u_{22}, 2)^4$};
\draw [fill=black] (8.8,1.2) circle (1.5pt);
\draw[color=black] (8.919867387322312,1.2106237921600767) node {$(u_{22}, 3)^4$};
\draw [fill=black] (9.,1.4) circle (1.5pt);
\draw[color=black] (9.027544472444113,1.3288816060014272) node {$(u_{23}, 1)^4$};
\draw [fill=black] (9.4,1.4) circle (1.5pt);
\draw[color=black] (9.421737185248617,1.3338090149114836) node {$(u_{23}, 2)^4$};
\draw [fill=black] (9.4,1.6) circle (1.5pt);
\draw[color=black] (9.475938683259235,1.6590180029751977) node {$(u_{23}, 3)^4$};
\draw [fill=black] (9.4,2.2) circle (1.5pt);
\draw[color=black] (9.406954958518448,2.1419040761607127) node {$(u_{31}, 1)^4$};
\draw [fill=black] (9.8,2.2) circle (1.5pt);
\draw[color=black] (9.820857306963177,2.146831485070769) node {$(u_{31}, 2)^4$};
\draw [fill=black] (9.8,2.4) circle (1.5pt);
\draw[color=black] (9.879986213883852,2.4523308374942583) node {$(u_{31}, 3)^4$};
\draw [fill=black] (9.8,2.6) circle (1.5pt);
\draw[color=black] (9.909841031703964,2.639572376076397) node {$(u_{32}, 1)^4$};
\draw [fill=black] (9.8,2.8) circle (1.5pt);
\draw[color=black] (9.820857306963177,2.9007250483093796) node {$(u_{32}, 2)^4$};
\draw [fill=black] (9.4,2.8) circle (1.5pt);
\draw[color=black] (9.41680977633856,2.895797639399323) node {$(u_{32}, 3)^4$};
\draw [fill=black] (9.,2.8) circle (1.5pt);
\draw[color=black] (9.012762245713944,2.9056524572194355) node {$(u_{33}, 1)^4$};
\draw [fill=black] (8.6,2.8) circle (1.5pt);
\draw[color=black] (8.613642123999384,2.9007250483093796) node {$(u_{33}, 2)^4$};
\draw [fill=black] (8.6,2.6) circle (1.5pt);
\draw[color=black] (8.563642123999384,2.540588651335609) node {$(u_{33}, 3)^4$};
\end{scriptsize}
\end{tikzpicture}
\end{center}



\begin{theorem}\label{m.2}
Let $G=LSG(n, m, k)$ be the layer Sun  graph which is defined already. Suppose that $n, m, k$ are integers such that $n, k\geq 3$ and $m\geq 2$.
Then,  the   cardinality of minimum doubly resolving set in  the line graph of the graph $G$ is   $nm^{k-2}-nm^{k-3}$.
\end{theorem}
\begin{proof}
Let $W$  be an ordered subset of the layer $U_k$ in the line graph of the graph $G$  such that
$$W=\{D^{(k)}_{1_1}-(u_{1_1}, 1)^k, D^{(k)}_{1_2}-(u_{1_2}, 1)^k, ..., D^{(k)}_{1_{m^{k-3}}}-(u_{1_{m^{k-3}}}, 1)^k;   ...; D^{(k)}_{n_1}-(u_{n_1}, 1)^k, D^{(k)}_{n_2}-(u_{n_2}, 1)^k, ..., D^{(k)}_{n_{m^{k-3}}}-(u_{n_{m^{k-3}}}, 1)^k\}.$$
Hence, $$V(L(G))- W=\{U_1, U_2, ..., U_{k-1}, (u_{1_1}, 1)^k, ..., (u_{1_{m^{k-3}}}, 1)^k, ...,(u_{n_1}, 1)^k, ..., (u_{n_{m^{k-3}}}, 1)^k \}.$$ We know that $|W|= nm^{k-2}-nm^{k-3}$. In a similar way as in Theorem \ref{f.1}, we can show that this subset is a minimal resolving set for the line graph of the graph $G$. We prove  that this subset
is  a doubly resolving set for the line graph of the graph $G$, and hence, $\beta(L(G)) =\psi(L(G))$. It is  sufficient  to prove that,  for any two vertices $u$ and $v$ in $L(G)$, there are vertices $x, y \in W$ such that $d(u, x) - d(u, y) \neq d(v, x) - d(v, y)$. Consider two vertices $u$ and $v$ in  $L(G)$. Then, we have the following: \newline

Case 1. Suppose that both vertices  $u$ and $v$ lie in the layer $U_1$. Hence,  there are  $r, s \in \{1, 2, ..., n\}$ such that $u=r$ and $v=s$. Moreover, we know that the line graph of the graph $G$ has the property that, for each vertex $r$ in  the layer $U_1$ there is some  vertex such as $x=(u_{r_j}, t)^k$  in the component $D^{(k)}_{r_j}$,  $1\leq j\leq m^{k-3}$   in  the layer $U_k$,  at distance $k-1$ from $u$, and in fact, $d(u, x)=k-1$.
In the same way, there is some  vertex such as $y=(u_{s_j}, t)^k$ in  the component $D^{(k)}_{s_j}$ in  the layer $U_k$,  at distance $k-1$ from $v$.  In particular, it is easy to prove that $d(u, x) - d(u, y)<0$  because $d(u, y)\geq k$. Also, $d(v, x) - d(v, y)>0$  because  $d(v, x)\geq k$.
\newline

Case 2. Now, suppose that both vertices  $u$ and $v$ lie in the layer $U_2$. In a similar way as in Case 1, we can show that there are vertices $x, y \in W$ such that $d(u, x) - d(u, y) \neq d(v, x) - d(v, y)$.\newline

Case 3. Suppose that both vertices  $u$ and $v$ lie in the layer $U_l$, $l\geq 3$ such that these vertices  lie in the  one component of the layer
$U_l$, say $D^{(l)}_{i_j}$, $1\leq i \leq n$, $1\leq j\leq m^{l-3}$.
In this case,  $d(u, v)=1$. Moreover, we know that the line graph of the graph $G$ has the property that, for each vertex
$u\in D^{(l)}_{i_j}$ in the layer $U_l$, there is a component of the layer $U_k$, say  $D^{(k)}_{i_r}$, $1\leq i\leq n$, $1\leq r \leq m^{k-3}$ such that, for every vertex $x\in D^{(k)}_{i_r}$, we have $d(u,x)=k-l$. In the same way, for the vertex $v\in D^{(l)}_{i_j}$ in  the layer $V_l$ there is a component of the layer $U_k$, say  $D^{(k)}_{i_s}$, $1\leq i\leq n$, $1\leq s \leq m^{k-3}$, $r\neq s$ such that, for every vertex $y\in D^{(k)}_{i_s}$, we have $d(v,y)=k-l$. Thus, $d(u, x) - d(u, y) \neq d(v, x) - d(v, y)$  because $d(u,y)=k-l+1$ and $d(v,x)=k-l+1$.\newline

Case 4. Suppose that both vertices  $u$ and $v$ lie in the layer $U_l$, $l\geq 3$ such that these vertices   lie in the  two distinct  components of the layer $U_l$. We can assume without loss of generality that $u\in D^{(l)}_{p_{j_1}}$ and $v\in D^{(l)}_{q_{j_2}}$,
$1\leq p, q \leq n$, and $1\leq {j_1}, {j_2}\leq m^{l-3}$.
Moreover, we know that the line graph of the graph $G$ has the property that, for each vertex $u\in D^{(l)}_{p_{j_1}}$ in  the layer $U_l$ there is a component of the layer $U_k$, say  $D^{(k)}_{p_r}$, $1\leq r \leq m^{k-3}$ such that for every vertex $x\in D^{(k)}_{p_r}$, we have $d(u,x)=k-l$. In the same way, for the vertex $v\in D^{(l)}_{q_{j_2}}$ in  the layer $V_l$, there is a component of the layer $U_k$, say  $D^{(k)}_{q_s}$, $1\leq s \leq m^{k-3}$,  such that, for every vertex $y\in D^{(k)}_{q_s}$, we have $d(v,y)=k-l$. In the following,
let two components $ D^{(l)}_{p_{j_1}}$, and $D^{(l)}_{q_{j_2}}$ are fundamental; indeed,  $p=q$. Hence $d(u, v)=2l-5$, $d(u, y)=d(v, x)=k+l-5$.
Thus, $d(u, x) - d(u, y) \neq d(v, x) - d(v, y)$.
Now, let two components $ D^{(l)}_{p_{j_1}}$  and $D^{(l)}_{q_{j_2}}$ are not fundamental; indeed,  $p\neq q$. Hence $d(u, v)=2l-4+d_{U_2}(u_p, u_q)$,
$d(u, y)=d(v, x)=k+l-4+d_{U_2}(u_p, u_q)$.
Thus, $d(u, x) - d(u, y) \neq d(v, x) - d(v, y)$.\newline

Case 5.
Suppose that  vertices  $u$ and $v$ lie in distinct layers $U_a, U_b$, respectively. Note that if $a=1$ and $b=2$, $a=1$ and $b>2$ or $a=2$ and $b>2$ there is nothing  to do. Now let, $3\leq a< b$. Hence, there is a  component of the layer $U_a$, say
$D^{(a)}_{i_j}$, $1\leq i\leq n$, $1\leq j \leq m^{a-3}$ such that $u\in D^{(a)}_{i_j}$.
Also, there is a  component of the layer $U_b$, say  $D^{(b)}_{p_q}$, $1\leq p\leq n$, $1\leq q \leq m^{b-3}$ such that $v\in D^{(b)}_{p_q}$.
In particular, there is a component of the layer $U_k$, say  $D^{(k)}_{i_r}$, $1\leq i\leq n$, $1\leq r \leq m^{k-3}$ such that, for any vertex
$x\in D^{(k)}_{i_r}$, we have $d(u, x)=k-a$. Now, let $i=p$;  if we consider
$y\in D^{(k)}_{z_s}$, $z\neq i$, $1\leq z\leq n$, and $1\leq s \leq m^{k-3}$, then we have $d(u, x) - d(u, y) \neq d(v, x) - d(v, y)$
because $d(u, y)=k+a-4+ d_{U_2}(u_i, u_z)$, $d(v, y)=k+b-4+ d_{U_2}(u_i, u_z)$, and $d(u, x) \neq d(v, x)$.
Note that if $i\neq p$, then  there is a component of the layer $U_k$, say  $D^{(k)}_{p_s}$, $1\leq p\leq n$, $1\leq s \leq m^{k-3}$ such that, for any vertex
$y\in D^{(k)}_{p_j}$, we have $d(v, y)=k-b$, and then, we have $d(u, x) - d(u, y) \neq d(v, x) - d(v, y)$.
\newline
Thus, from   the above-mentioned  cases we  conclude that the  cardinality of minimum doubly
resolving set in the line graph of the graph $G$   is $nm^{k-2}-nm^{k-3}$.\newline
\end{proof}
\begin{theorem}\label{m.2}
Let $G=LSG(n, m, k)$ be the layer Sun  graph which is defined already. Suppose that $n, m, k$ are integers such that $n, k\geq 3$ and $m\geq 2$. Then, the strong metric dimension  in  the line graph of the graph $G$ is $nm^{k-2}-1$.
\end{theorem}
\begin{proof}
In a similar way which is done in the proof of Theorem \ref{f.3},  we can show that  the subset
$$W=\{D^{(k)}_{1_1}-(v_{1_1}, 1)^k, D^{(k)}_{1_2}, ..., D^{(k)}_{1_{m^{k-3}}},   ..., D^{(k)}_{n_1}, D^{(k)}_{n_2}, ..., D^{(k)}_{n_{m^{k-3}}}\},$$
of vertices in the line graph of the graph $G$ is  a minimal resolving set of $L(G)$.
\end{proof}
\section{Conclusion}
In this paper, we have constructed a layer Sun  graph $LSG(n, m, k)$, discussed of this graph, and computed  the minimum cardinality of  doubly resolving set and strong resolving set  of layer Sun  graph $LSG(n, m, k)$  and the line graph of the layer Sun  graph $LSG(n, m, k)$.
We deduce that, by this way, we can construct a layer jellyfish graph $JFG(n, m, k)$, of order $n+ \Sigma_{r=1}^{k-1}nm^{r}$,
where the jellyfish graph $JFG(n, m)$, which is defined in [5], and  by a similar way, we can obtain and compute the minimum cardinality of   doubly resolving set  and strong resolving set of layer jellyfish graph $JFG(n, m, k)$ and the line graph of the layer jellyfish graph $JFG(n, m, k)$.
\newline

\bigskip
{\footnotesize
\noindent \textbf{Data Availability}\\
No data were used to support this study.\\[2mm]
\noindent \textbf{Conflicts of Interest}\\
The authors declare that there are no conflicts of interest
regarding the publication of this paper.\\[2mm]
\noindent \textbf{Acknowledgements}\\
The work was partially supported by the Project of Anhui
Jianzhu University under grant nos. 2016QD116 and
2017dc03. \\[2mm]
\noindent \textbf{Authors' informations}\\
\noindent Jia-Bao Liu${}^{a,b}$
(\url{liujiabaoad@163.com;liujiabao@ahjzu.edu.cn})\\
Ali Zafari${}^{c}$(\textsc{Corresponding Author})
(\url{zafari.math.pu@gmail.com}; \url{zafari.math@pnu.ac.ir})\\

\noindent
${}^{a}$ Department of Mathematics, Huainan Normal University, Huainan 232038,  P.R. China.\\
${}^{b}$ School of Mathematics and Physics, Anhui Jianzhu University, Hefei 230601, P.R. China.\\
${}^{c}$ Department of Mathematics, Faculty of Science,
Payame Noor University, P.O. Box 19395-4697, Tehran, Iran.
{\footnotesize


\bigskip
\end{document}